\documentclass[reqno, a4paper]{amsart}

\usepackage[english]{babel} 

\usepackage{amsmath}
\usepackage{amsfonts}
\usepackage{amssymb}
\usepackage{amsthm}
\usepackage{tikz}
\usepackage{mathrsfs}
\usepackage{mathtools}
\usepackage{enumitem}
\usepackage{url}

\newcommand{\longsurj}{%
\relbar\joinrel\twoheadrightarrow
}

\newcommand{\N}{\mathbb{N}}
\newcommand{\R}{\mathbb{R}}
\renewcommand{\R}{\mathbb{R}}

\renewcommand{\leq}{\leqslant}
\renewcommand{\geq}{\geqslant}
\newcommand{\seq}{\subseteq}
\newcommand{\df}{\coloneqq}

\newcommand{\C}{{\sf C}}
\renewcommand{\P}{\mathscr{P}}

\newcommand{\HAf}{{\sf HA}_{\sf f}}
\newcommand{\Posf}{{\sf Pos}_{\sf f}}

\DeclareMathOperator{\Log}{{\sf Log}}
\newcommand{\BD}{\text{\sc bd}}
\newcommand{\gen}{\reflectbox{$\neg$}}
\DeclareMathOperator{\Op}{\mathscr{O}}
\DeclareMathOperator{\K}{\mathscr{C}}
\DeclareMathOperator{\Subc}{\mathrm{Sub_c}}
\DeclareMathOperator{\Subo}{\mathrm{Sub_o}}

\DeclareMathOperator{\conv}{\rm conv}
\DeclareMathOperator{\inte}{\rm int}
\DeclareMathOperator{\depth}{\rm dep}
\DeclareMathOperator{\Up}{\rm Up}
\DeclareMathOperator{\Lo}{\rm Lo}
\DeclareMathOperator{\Spec}{\rm Spec}
\newcommand{\gammaup}{\gamma^{\uparrow}}

\newcommand{\face}{\preccurlyeq}
\newcommand{\faceneq}{\prec}
\DeclareMathOperator{\cl}{\mathrm{cl}}
\DeclareMathOperator{\rint}{\mathrm{relint}}
\DeclareMathOperator{\PCc}{{\sf P}_{c}} % polyconvex sets
\DeclareMathOperator{\PCo}{{\sf P}_{o}} % open polyconvex sets
\newcommand{\ostar}[1]{\mathrm{o}(#1)}
\newcommand{\Polyd}{{\sf P}_{d}}
\DeclareMathOperator{\nerve}{{\mathscr{N}}}
\DeclareMathOperator{\geo}{{\nabla}}

\newtheorem{thm}{Theorem}[section]
\newtheorem{lem}[thm]{Lemma}
\newtheorem{cor}[thm]{Corollary}

\theoremstyle{definition}
\newtheorem{defn}[thm]{Definition}
\newtheorem*{construction}{Construction}

\theoremstyle{remark}
\newtheorem{rem}[thm]{Remark}

\theoremstyle{remark}
\newtheorem{ex}[thm]{Example}

\theoremstyle{theorem}
\newtheorem*{mainthm}{Theorem}

\title[Tarski's Theorem and Polyhedra]{Tarski's Theorem on Intuitionistic Logic, for Polyhedra}
\author[N. Bezhanishvili]{Nick Bezhanishvili}
\author[V. Marra]{Vincenzo Marra}
\author[D. McNeill]{Daniel McNeill}
\author[A. Pedrini]{Andrea Pedrini}
\address[N. Bezhanishvili]{Institute for Logic, Language and Computation,
University of Amsterdam, P.O. Box 94242, 1090 GE Amsterdam, The Netherlands}
\address[V. Marra]{Dipartimento di Matematica {\sl Federigo Enriques}, Universit\`a degli Studi di Milano, via Cesare Saldini 50, 20133 Milano, Italy.}
\address[A. Pedrini]{Dipartimento di Matematica {\sl Felice Casorati}, Universit\`a degli Studi di Pavia, Via Ferrata 5, 27100 Pavia, Italy.}
\address[D. McNeill]{Dipartimento di Scienze Teoriche e Applicate, Universit\`a degli Studi dell'Insubria, Via Mazzini 5, 21100 Varese, Italy.}
\email[N. Bezhanishvili]{n.bezhanishvili@uva.nl}
\email[V. Marra]{vincenzo.marra@unimi.it}
\email[D. McNeill]{danmcne@gmail.com}
\email[A. Pedrini]{andrea.pedrini@unipv.it}

\thanks{2010 {\it Mathematics Subject Classification.
}
Primary: 03B20. Secondary: 06D20; 06D22; 55U10; 52B70; 57Q99.}
\keywords{Intuitionistic logic; topological semantics; completeness theorem; finite model property; Heyting algebra; locally finite algebra; polyhedron; simplicial comple; triangulation; PL topology.}
\begin{document}
\maketitle

\begin{abstract}In 1938, Tarski proved that a formula is not intuitionistically valid if, and only if, it has a counter-model in the Heyting algebra of open sets of some topological space. In fact, Tarski showed that any Euclidean space $\R^{n}$ with $n\geq 1$ suffices, as does e.g.\  the Cantor space. In particular, intuitionistic logic cannot detect topological dimension in the frame of all open sets of a Euclidean space. By contrast, we consider the lattice of open subpolyhedra of a given compact polyhedron $P\seq \R^{n}$, prove that it is a locally finite Heyting subalgebra of the (non-locally-finite) algebra of all open sets of $\R^{n}$, and show that intuitionistic logic is able to capture the topological dimension of $P$ through the bounded-depth axiom schemata. Further, we show that intuitionistic logic is precisely the logic of formul\ae\ valid in all Heyting algebras arising from polyhedra in this manner. Thus, our main theorem  reconciles through polyhedral geometry two classical results:  topological  completeness in the style of Tarski, and Ja\'skowski's theorem that intuitionistic logic enjoys the finite model property. Several questions of interest remain open. E.g., what is the intermediate logic of all closed triangulable manifolds?
\end{abstract}

\section{Introduction}\label{s:intro}
If $X$ is any topological space, the collection $\Op{(X)}$ of its open subsets is a (complete) Heyting algebra whose underlying order is given by set-theoretic inclusion. One can then interpret formul\ae\ of intuitionistic logic into $\Op{(X)}$ by  assigning open sets to propositional atoms, and then extending the assignment to formul\ae\ using the  operations of the Heyting algebra $\Op{(X)}$.  A formula is true under such an interpretation just when it evaluates to $X$. In 1938, Tarski (\cite{Tarski38}, English translation in \cite{LSM2nd}) proved that intuitionistic logic is complete with respect to this semantics. Moreover, Tarski showed that one can considerably restrict the class $\C$ of spaces under consideration without impairing completeness. In particular,  one can take $\C\df\{X\mid \text{ $X$ is metrisable}\}$, and even $\C\df\{\R\}$ or $\C\df\{2^{\N}\}$, where $2^{\N}$ denotes the Cantor space. Tarski's result opened up a research area that continues to prosper to this day. Immediate descendants of \cite{Tarski38} are the three seminal papers \cite{McKinseyTarski44, McKinseyTarski46, McKinseyTarski48} by McKinsey and Tarski; \cite[\S 3]{McKinseyTarski46} offers a different proof of the main result of \cite{Tarski38} in the dual language of closed sets and co-Heyting  algebras.
For an exposition of the different themes in spatial logic we refer to \cite{APvB07}.

Intuitionistic logic has the finite model property. In  1936 Ja\'skowski sketched a proof of this fact  \cite{Jaskowski36};  the first detailed exposition of the result\footnote{Though not exactly of the proof sketched by Ja\'skowski: cf.\ \cite[Lemma 5.3 and footnote (16)]{Rose53}.} seems to be \cite[Theorem 5.4]{Rose53} (see also \cite[Theorem 2.57]{CZ97}).
Algebraically, the finite model property may be rephrased into the statement that there exists a set of finite Heyting algebras that generates the equational class (or \emph{variety}) of all Heyting algebras. An algebraic proof of this result was first obtained by McKinsey and Tarski \cite{McK41, McKinseyTarski44} (see \cite{BB16} for a discussion 
of this  proof and  a comparison with the model-theoretic method of filtration). Ja\'skowski's proof shows that, in fact, there is a countable, recursively enumerable\footnote{Each finite Heyting algebra being presented, e.g., by the finite multiplication tables for its operations. Ja\'skowski's theorem yields at once the decidability of intuitionistic logic. 
More is  known: the problem of deciding whether a formula is intuitionistically provable is {\sc pspace}-complete \cite{Statman79}.} such set.

Recall that an algebraic structure is \emph{locally finite} if its finitely generated substructures are finite. The Heyting algebras $\Op{(\R)}$ and $\Op{(2^{\N})}$ are very far from being locally finite. For example, \cite[Theorem 3.33]{McKinseyTarski46} shows that any Heyting algebra freely generated by a finite set embeds into both $\Op{(\R)}$ and $\Op{(2^{\N})}$, and  already the Heyting algebra freely generated by one element (the Rieger-Nishimura lattice \cite{Rieger49, Nishimura62}) is of course infinite. Thus, while counter-models to formul\ae\ that are not intuitionistically provable always exist  in $\Op{(\R)}$, they are not automatically finite: one has to pick the open sets to be assigned to atomic formul\ae\ with extra care in order to exhibit a  finite counter-model such as the ones guaranteed by the finite model property, see e.g., \cite{BG05}.

Our main result provides a theorem in the style of Tarski that has the advantage of using locally finite Heyting algebras of opens sets only, and hence affords at the same time the advantages of Ja\'skowski's theorem. Our result exposes and exploits, we believe for the first time, the connection between intuitionistic logic and the classical  PL (=piecewise linear) category of compact polyhedra in Euclidean spaces \cite{RourkeSanderson82, Maunder80}. The needed background is recalled in Section \ref{s:pre}, to which the reader is referred for all unexplained notions in the rest of this Introduction. To state our  results we prepare some notation.

For each $n\in\N\coloneqq\left\{0,1,2,\ldots\right\}$ and each (always compact) polyhedron $P\seq\R^{n}$, we write $\Subc{P}$ for the collection of subpolyhedra of $P$ --- i.e., polyhedra in $\R^{n}$ contained in $P$. We set
\[
\Subo{P}:=\left\{O\subseteq P \mid P\setminus O \in \Subc{P}\right\},
\]
where $\setminus$ is set-theoretic difference. Members of $\Subo{P}$ are called \emph{open \textup{(}sub\textup{)}polyhedra} (of $P$).
It is a standard fact that $\Subo{P}$   is a distributive lattice under set-theoretic intersections and unions, and hence a sublattice of\footnote{Here and throughout, $P$ is always equipped with the subspace topology inherited from the Euclidean topology of $\R^{n}$.} $\Op{(P)}$.  In Section \ref{s:onepoly} we prove that $\Subo{P}$ is, in fact, a Heyting subalgebra of $\Op{(P)}$. In the same section we prove that, unlike  $\Op{(P)}$, $\Subo{P}$ is always locally finite. The proof provides  one of the key  insights of the present paper: local finiteness  essentially amounts to the Triangulation Lemma of PL topology,  and thus reflects algebraically a crucial tameness property of polyhedra as opposed to general compact subsets of $\R^{n}$. 

Further tameness properties of polyhedra emerge from their dimension theory, which is far simpler than the dimension theory of  general metric spaces. All standard topological dimension theories  agree on polyhedra \cite{HurewiczWallman41, Pears75}. In fact, an elementary notion of dimension is available for every nonempty polyhedron $\emptyset\neq P$, in that $\dim{P}\leq d$ holds if, and only if, any $d+2$ distinct points of $P$ are affinely dependent. In Section \ref{s:dim}, we  establish a fundamental connection between the topological dimension of $P$ and the structure of $\Subo{P}$: the latter lies in the variety of Heyting algebras of bounded depth $d$ if, and only if, $\dim{P}\leq d$. Recall that the \emph{bounded-depth axiom schemata} (see, e.g., \cite[Section 2.5]{CZ97}) are inductively defined as follows, over the countably infinite set $\{\alpha_{0},\ldots,\alpha_{n},\ldots\}$ of  propositions:
\[
 \BD_{d}\coloneqq
 \begin{cases}
\left(\,\alpha_0 \vee \neg \alpha_0 \,\right) & \text{ if $d=0$, and}\\
\left(\,\alpha_{d} \vee (\alpha_{d} \to \BD_{d-1})\,\right) & \text{ if $d\geq 1$.}
 \end{cases}
\]
If now $\P$ is any  family of polyhedra, we write $\Log{\P}$ for the extension of intuitionistic logic determined by $\P$, namely, the unique intermediate logic corresponding to the variety of Heyting algebras generated by the collection of Heyting algebras $\left\{\Subo{P}\mid P\in\P\right\}$. 

 Given $d\in\N$, let us denote by $\Polyd$ the set of all polyhedra of dimension less than or equal to $d$. Consider   any finite poset $A$ of depth $d\in\N$. In Section \ref{s:nerves}, using  Alexandrov's notion of  nerve \cite{Alexandrov98}, we construct a polyhedron $P$ of dimension $d$  such that the Heyting algebra of upper sets of $A$ embeds into the Heyting algebra $\Subo{P}$. This  leads to our main result:
\begin{mainthm}\label{t:A}For each $d\in\N$,
$\Log{\Polyd}$ is intuitionistic logic extended by the axiom schema $\BD_{d}$. Hence, the logic  $\Log{\,\bigcup_{d\in\N} \Polyd\,}$ of all polyhedra is  intuitionistic logic.
\end{mainthm}
We prove the theorem in Section \ref{s:main}. Our proof is self-contained to within the standard facts from PL topology and Heyting algebras recalled in Section \ref{s:pre}.

Returning to Tarski's theorem, let us consider Euclidean spaces $\R^{N}$ and $\R^{n}$ with $N>n\in\N$. In line with the compact setting of the present paper, let us in fact confine attention to their unit cubes $[0,1]^{N}$ and $[0,1]^{n}$. Then Tarski's results show, {\it inter alia}, that the Heyting algebras $\Op{([0,1]^{N})}$ and $\Op{([0,1]^{n})}$ satisfy precisely the same equations --- i.e., in both cases the corresponding logic is intuitionistic logic --- regardless of the fact that one cube has strictly larger topological dimension than the other.  However, if we consider the smaller Heyting algebras of open subpolyhedra of the two cubes,
then $\BD_{n}$  is valid in $[0,1]^{n}$ and  is refuted in $[0,1]^{N}$. Restriction to a class of  tame, geometric subsets of Euclidean space such as the polyhedra of our paper 
thus allows us to express the dimension of Euclidean spaces by means of intuitionistic logic  and Heyting algebras.
\section{Preliminaries}\label{s:pre}
We assume familiarity with intuitionistic logic and Heyting algebras. A few standard references are \cite{BalbesDwinger74, Johnstone82, MacLaneMoerdijk94, CZ97}. 
In this section we recall what we need. On the other hand, we assume rather less about PL topology. All needed definitions and results are recalled in detail in this section. A few standard  references are \cite{Stallings67, Hudson69, Glaser70, Maunder80, RourkeSanderson82}.

`Distributive lattice' means `bounded distributive lattice'; homomorphisms are to preserve both the maximum ($\top$) and the minimum ($\bot$) element. We write $\wedge$ and $\vee$ for meets and joins, and write $\to$ and $\neg$ for  Heyting implication and negation.
\subsection{Posets, frames and p-morphisms}\label{ss:pmor}
We denote the partial order relation on any poset by $\leq$, unless otherwise specified. Given any poset $A$ and any $a\in A$, we set 
\begin{align*}
\uparrow{}a &\df\left\{x \in  A \mid a\leq x\right\},\\
\downarrow{}a &\df\left\{x \in  A \mid x\leq a\right\}.
\end{align*}
An \emph{upper set} in $A$ is a subset $U\seq A$ closed under $\uparrow$: if $a\in A$ satisfies $a\in U$, then $\uparrow{}a\seq U$. Similarly, a \emph{lower set} in $A$ is a subset  closed under $\downarrow$.
A \emph{chain} is a totally ordered set. A \emph{chain in $A$} is a subset $C\seq A$ that is a chain when equipped with the order inherited from $A$. We  define the \emph{depth} of $A$ to be
\[
\depth{A}\df\sup{\left\{|C|-1\, \mid\,  C \seq A \text{ is a chain in $A$} \right\}}\in \N\cup\left\{\infty\right\}.
\]
If $A$ and $B$ are posets, a \emph{p-morphism} from $A$ to $B$ is an order-preserving function $f\colon A\to B$ that commutes with $\uparrow$: for each $a\in A$,
\[
f[\uparrow{}a]=\uparrow{}f(a).
\]
Here and throughout, $f[\cdot]$ denotes direct image under the function $f$. Similarly, $f^{-1}[\cdot]$ will denote inverse image under the function $f$.

An \emph{\textup{(}intuitionistic Kripke\textup{)} frame} is just a poset. It is \emph{rooted} if it has a minimum. Any frame $A$ gives rise to a Heyting algebra. First, set
\[
\Up{A}\df\left\{U\seq A \mid \text{$U$ is an upper set in $A$}\right\}.
\]
Under the inclusion order, $\Up{A}$ is a complete distributive lattice; arbitrary meets and joins are provided by set-theoretic unions and intersections. Hence the meet operation has an adjoint, the uniquely determined  implication of $\Up{A}$ that makes it into a Heyting algebra.\footnote{See Subsection \ref{ss:frame} for the generalisation of this construction to all topological spaces.} For later use in the paper, we also prepare the dual notation
\[
\Lo{A}\df\left\{L\seq A \mid \text{$L$ is lower set in $A$}\right\}.
\]
As for $\Up{A}$, we will always regard $\Lo{A}$ as a complete distributive lattice under the inclusion order. $\Lo{A}$ has a uniquely determined co-Heyting algebraic structure.\footnote{As well as a Heyting one that will not be used in this paper.}

Conversely, we can associate a poset to any Heyting algebra $H$. Set
\[
\Spec{H}\df\left\{F\seq H\mid \text{$F$ is a prime filter} \right\}.
\]
Here, we mean that $F$ is a prime filter of the underlying distributive lattice of $H$. Equipping $\Spec{H}$ with the inclusion order, we obtain a poset. 

The Heyting algebras of the form $\Up{A}$, as $A$ ranges over all finite posets, are precisely the finite Heyting algebras. To see this, given a Heyting algebra $H$, we consider the \emph{Stone map}:
\begin{align}
\widehat{\cdot}\, \colon H&\longrightarrow \Up{\Spec{H}}\label{eq:stone}\\
h\in H&\longmapsto \widehat{h}\df\left\{\mathfrak{p}\in\Spec{H}\mid h\in\mathfrak{p}\right\}.\nonumber
\end{align}
The following goes back to \cite{Birkhoff37}.
\begin{lem}\label{l:birkhoff}
For any finite Heyting algebra $H$, the Stone map \eqref{eq:stone} 
is an isomorphism of Heyting algebras.
\end{lem}
\begin{proof} For detailed proofs see \cite[Sec.~8.4]{CZ97} and \cite{Morandi05}.
\end{proof}
With the above in place, a modern statement of a part of Ja\'skowski's  result cited in the Introduction is:
\begin{lem}[The finite model property]\label{l:fmp}The equational class of Heyting algebras is generated by the finite Heyting algebras: any Heyting algebra is a homomorphic image of a subalgebra of a product of  Heyting algebras of the form $\Up{A}$, as $A$ ranges over all  finite posets. 
\end{lem}
\begin{proof} See e.g.\ \cite[Thms.~2.57 and 7.21]{CZ97}. 
\end{proof}
\begin{rem} One can restrict the class of finite posets featuring in  Lemma \ref{l:fmp} in various ways. Thus, Ja\'skowski exhibited a specific recursive sequence of posets. It is also known, for instance, that the class of all finite trees (=rooted frames $T$ such that $\downarrow{}t$ is a chain for each $t\in T$) suffices, see e.g., \cite[Cor.~2.33 and Ex.~2.17]{CZ97}.
 In this paper we only need the general form of the result as stated in Lemma \ref{l:fmp}.\qed
\end{rem}

\subsection{Finite Esakia duality}\label{ss:esakia}
Lemma \ref{l:birkhoff} can be lifted to a contravariant equivalence of categories between Heyting algebras and Esakia spaces \cite{Esakia74}. In the finite case of interest here, topology can and will be dispensed with. Given a homomorphism of finite Heyting algebras $h\colon H \to K$,
set
\begin{align*}
\Spec{h} \colon \Spec{K} &\longrightarrow\Spec{H}\\
\mathfrak{p}\in\Spec{K} &\longmapsto h^{-1}[\mathfrak{p}]\in\Spec{H}.
\end{align*}
Dually, given a $p$-morphism of posets $f\colon A\to B$, set
\begin{align*}
\Up{f} \colon \Up{B} &\longrightarrow\Up{A}\\
U\in\Up{B} &\longmapsto f^{-1}[U]\in\Up{A}.
\end{align*}
Let now $\HAf$ and $\Posf$ denote the categories of finite Heyting algebras and their homomorphisms, and of finite posets and p-morphisms, respectively. Then the above defines functors
\begin{align*}
\Spec \colon \HAf &\longrightarrow\Posf^{\rm op},\\
\Up \colon \Posf &\longrightarrow \HAf^{\rm op}.
\end{align*}
(We are indicating by ${\sf C}^{\rm op}$ the category opposite to the category {\sf C}, as is standard.)
\begin{lem}[Esakia duality, finite case]\label{l:esakia} The functors $\Spec$ and $\Up$ are an equivalence of categories.
\end{lem}
\begin{proof}
See \cite[Exs.~7.5, 7.6 and Sec.~8.5]{CZ97} and \cite{Morandi05}.
\end{proof}
\begin{rem}As with all duality results, Lemma \ref{l:esakia} provides a dictionary between notions in $\HAf$ and notions in $\Posf$. For example, one shows that a surjective p-morphism of finite posets dualises to an injective homomorphism of finite Heyting algebras, i.e.\ to a Heyting subalgebra, and conversely. We do not dwell on the  details of such translations, and use them whenever needed in the sequel. \qed
\end{rem}

\subsection{Bounded depth}\label{ss:bd}Through the equations corresponding to the formul\ae\ $\BD_{d}$ of the Introduction,
one can  express equationally the analogue for Heyting algebras of the Krull dimension of commutative rings.\footnote{For recent related literature see  \cite{BBL-BvM16}, where a Krull dimension is defined for any topological space and is used in obtaining fine-grained topological completeness results for modal and intermediate logics.}
\begin{lem}\label{l:bd}For any non-trivial Heyting algebra $H$ and each $d\in \N$, the following are equivalent.
\begin{enumerate}[label={\rm (\roman*)}]
\item The longest chain of prime filters in $H$ has cardinality $d+1$.
\item $\depth{\Spec{A}}=d$.
\item $H$ satisfies the equation $\BD_{d}=\top$, and fails each equation $\BD_{d'}=\top$ with $1\leq d'<d$.
\end{enumerate}
\end{lem}
\begin{proof} See \cite[Prop.~2.38 and Table 9.7]{CZ97} and  \cite{BBL-BvM16}.
\end{proof}

\subsection{Heyting and co-Heyting algebras of open and closed sets.}\label{ss:frame}The open (closed) sets of a topological space provide important examples of (co-)Heyting algebras. For background on co-Heyting algebras we refer to \cite[\S 1 and {\it passim}]{McKinseyTarski46}, where these structures were first axiomatised equationally, and systematically investigated under the name of `Brouwerian algebras'. We write $\gen$ to denote co-Heyting negation, and $\leftarrow$ to denote co-Heyting implication.\footnote{McKinsey's and Taski's original notations were $\neg$ and $\stackrel{_{.}}{-}$, respectively.}

 If $X$ is any topological space, we write $\Op{(X)}$ for its collection of opens sets. Then $\Op{(X)}$ is a complete distributive lattice, bounded above by $\top\df X$ and below by $\bot\df\emptyset$, with  joins given by set-theoretic unions and meets given by
\[
\bigwedge F\df \inte{\,\bigcap F}
\]
for any family $F$ of open subsets of $X$, where $\inte$ denotes the interior operator of the given topology on $X$. Therefore $\Op{(X)}$ has exactly one structure of Heyting algebra compatible with its distributive-lattice structure; namely, for any $U,V\in\Op{(X)}$ the Heyting implication is given by
\begin{equation}\label{e:hi} 
  U \to V \df \bigcup \left\{O\in \Op{(X)} \mid U\cap O \subseteq V\right\} =\inte{(\,(X\setminus U)\cup V)\,)}.
 \end{equation}
 In particular, the Heyting negation is given by
 \begin{equation*}
  \neg U \df U \to \bot =\inte{(X \setminus U)}.
 \end{equation*} 
Dually, the family $\K{(X)}$ of closed sets of $X$ is a complete distributive lattice, bounded above by $\top\df X$ and below by $\bot\df\emptyset$, with  meets given by set-theoretic intersections and joins given by
\[
\bigvee F\df \cl{\,\bigcup F}
\]
for any family $F$ of closed subsets of $X$, where $\cl$ denotes the closure operator of the given topology on $X$.  Therefore $\K{(X)}$ has exactly one structure of co-Heyting algebra compatible with its distributive-lattice structure; namely, for any $C,D\in\K{(X)}$ the co-Heyting implication is given by
\begin{equation}\label{e:ci}
  C \leftarrow D \df \bigcap \left\{K\in \K{(X)} \mid C \seq D\cup K\right\} = \cl{(C\setminus D)}.
 \end{equation} 
 In particular, the co-Heyting negation is given by
 \begin{equation*}
  \gen D \df \top \leftarrow D = \cl{(X \setminus D)}.
 \end{equation*} 
 \begin{rem}All our results in this paper have versions for Heyting and co-Heyting algebras. We stressed the Heyting version in the Introduction, as this relates most directly to inutitionistic logic. However, we will see below that it is at times convenient in proofs to establish the co-Heyting version of the results first, because it is traditional in simplicial topology to work with closed simplices and polyhedra. Proofs for the corresponding Heyting versions are obtained through dual arguments, which we sometime omit.
 \qed
 \end{rem}

\subsection{Polyhedra: basic notions}\label{ss:polybasic}
An \emph{affine combination} of $x_0,\dots,x_d\in\R^n$ is an element $\sum_{i=0}^{d}r_{i}x_{i}\in \R^{n}$, where $r_{i}\in\R$ and $\sum_{i=0}^{d}r_{i}=1$. If, additionally, $r_{i}\geq 0$ for each $i\in\{0,\ldots,d\}$, $\sum_{i=0}^{d}r_{i}x_{i}$ is a \emph{convex combination}.
Given any  subset $S\seq\R^{n}$, the \emph{convex hull} of $S$, written $\conv{S}$, is the collection of all convex combinations of finite subsets of $S$. Then $S$ is \emph{convex} if $S=\conv{S}$, and a \emph{polytope} if $S=\conv{V}$ for a finite set $V\seq\R^n$. A \emph{polyhedron} in $\R^{n}$ is any subset  that can be written as a finite union of polytopes. The union over an empty index set is allowed, so that $\emptyset$ is a polyhedron. Any polyhedron is closed and bounded, hence compact. If $P\seq \R^{n}$ is a polyhedron, by an \emph{open polyhedron} in $P$ we mean the complement of a polyhedron which is included in $P$.
The points $x_0,\dots,x_d\in\R^n$ are \emph{affinely independent} if  the vectors $x_1-x_0, x_2-x_0,\dots,x_d-x_0$ are linearly independent, a condition which is invariant under  permutations of the index set $\{0,\ldots,d\}$. A \emph{simplex} in $\R^{n}$ is a \emph{non-empty}\footnote{It is expedient in this paper not to regard $\emptyset$ as a simplex.} subset of the form $\sigma\coloneqq\conv{V}$, where $V\coloneqq\{x_{0},\ldots,x_{d}\}$ is a set of affinely independent points. Then $V$ is the uniquely determined such affinely independent set \cite[Proposition 2.3.3]{Maunder80}, and $\sigma$ is a \emph{$d$-simplex} with \emph{vertices} $x_{0},\ldots,x_{d}$. A \emph{face} of the simplex $\sigma$ is the convex hull of a non-empty subset of $V$, and thus is itself a $d'$-simplex for a uniquely determined $d'\in\{0,\ldots,d\}$. Hence the $0$-faces of $\sigma$ are precisely its vertices.

We write 
\begin{align*}
\sigma&=x_{0}\cdots x_{d}, \,\,\, \sigma\face\tau,\, \text{ and } \sigma\faceneq\tau
\end{align*}
to indicate that $\sigma$ is the $d$-simplex whose vertices are $x_{0},\ldots,x_{d}$, that $\sigma$ is a face of $\tau$, and that $\sigma$ is a proper (i.e.\ $\neq\tau$) face of $\tau$, respectively. 
If $\sigma=x_{0}\cdots x_{d}\in\R^{n}$, the \emph{relative interior} of $\sigma$, denoted $\rint{\sigma}$, is the topological interior of $\sigma$ in the affine subspace of $\R^{n}$ spanned by\footnote{Recall that the affine subspace spanned by a subset $S\seq\R^{n}$  is the collection of all affine combinations of finite subsets of $S$, or equivalently, the intersection of all affine subspaces of $\R^{n}$ containing $S$.} $\sigma$.  (Thus, the relative interior of a $0$-dimensional simplex --- a point --- is the point itself.)
 To rephrase through  coordinates, note that by the affine independence of the vertices of $\sigma$, for each $x\in\sigma$ there exists a unique choice of $r_{i}\in\R$ with $x=\sum_{i=0}^{d}r_{i}x_{i}$ and $r_i\geq 0$, $\sum_{i=1}^{d} r_{i}=1$.  The $r_{i}$'s are traditionally called the \emph{barycentric coordinates} of $x$. Then $\rint{\sigma}$ coincides with  the subset  of $\sigma$ of those points $x\in\sigma$ whose barycentric coordinates are strictly positive. Note that $\cl{\rint{\sigma}}=\sigma$, the closure being taken in the ambient Euclidean space $\R^n$. In the rest of this paper, for any set $S\seq \R^n$ we use the notation \begin{align}\label{eq:closiurenot}
 \cl{S}
 \end{align}
 to denote the closure of $S$ in the ambient Euclidean space $\R^n$. Observe that if $P\seq\R^n$ is a polyhedron and $S\seq P$, then the closure of $S$ in the subspace $P$ of $\R^n$ agrees with $\cl{S}$, because $P$ is closed in $\R^n$.
\subsection{Polyhedra: the Triangulation Lemma}\label{ss:tl}
\begin{defn}[Triangulation]\label{d:triang} A \emph{triangulation}\footnote{Also known as (\emph{geometric}) \emph{simplicial complex}. Note that the empty triangulation $\emptyset$ is allowed.}  is a finite set $\Sigma$ of simplices in $\R^{n}$ satisfying the following conditions.
\begin{enumerate}
 \item\label{d:t1} If $\sigma\in \Sigma$ and $\tau$ is a face of $\sigma$, then $\tau\in \Sigma$.
 \item\label{d:t2} If $\sigma,\tau\in \Sigma$, then $\sigma \cap \tau$ is either empty, or a common face of $\sigma$ and $\tau$.
\end{enumerate}
The \emph{support}, or \emph{underlying polyhedron}, of the triangulation $\Sigma$ is 
\[
|\Sigma|\coloneqq \bigcup\Sigma \ \seq \R^{n}.
\]
One also says that $\Sigma$ \emph{triangulates} the subset $|\Sigma|$ of $\R^{n}$. A \emph{subtriangulation} of the triangulation $\Sigma$ is any subset $\Delta\seq\Sigma$ that is itself a triangulation. This is equivalent to the condition that $\Delta$ be closed under taking faces --- i.e.\ satisfies just  \eqref{d:t1} in Definition \ref{d:triang} --- for then \eqref{d:t2} follows \cite[Proposition 2.3.6]{Maunder80}. By the \emph{vertices} of $\Sigma$ we mean the vertices of the simplices in $\Sigma$.
\end{defn}

Observe that a subtriangulation of $\Sigma$ is precisely the same thing as a lower set of $\Sigma$, the latter being regarded as a poset under inclusion. This fact will be heavily exploited below, cf.\ in particular Section \ref{s:dim}.
The following standard fact makes precise the idea that a triangulation $|\Sigma|$ provides a finitary description of the triangulated space $\Sigma$.
\begin{lem}
\label{l:interior}
 If $\Sigma$ is a triangulation,  for each  $x\in|\Sigma|$ there is exactly one simplex $\sigma^x\in\Sigma$ such that $x\in\rint{\sigma}$.
\end{lem}
\begin{proof}See \cite[Proposition 2.3.6]{Maunder80}.
\end{proof}
In light of Lemma \ref{l:interior}, in the sequel we adopt the notation $\sigma^{x}$ without further comment; the simplex $\sigma^{x}$ is called the \emph{carrier} of $x$ (in $\Sigma$).

Any subset of $\R^{n}$ that admits a triangulation, being a finite union of simplices, is evidently a polyhedron. The rather less trivial converse is true, too, in the following strong sense.
\begin{lem}[Triangulation Lemma]\label{l:tl}Given finitely many polyhedra $P, P_{1},\ldots, P_{m}$ in $\R^{n}$ with $P_{i}\seq P$ for each $i\in\{1,\ldots,m\}$, there exists a triangulation $\Sigma$ of $P$ such that, for each $i\in\{1,\ldots,m\}$, the collection
\[
\Sigma_{i}\df\left\{\sigma \in \Sigma \mid \sigma \seq P_{i} \right\}
\]
is a triangulation of $P_{i}$, i.e.\ $|\Sigma_{i}|=P_{i}$.
\end{lem}
\begin{proof}\cite[Theorem 2.11 and Addendum 2.12]{RourkeSanderson82}.
\end{proof}

The Triangulation Lemma  is the fundamental tool in this paper. Recall from  the Introduction that $\Subc{P}$ and $\Subo{P}$ denote the collections of polyhedra and open polyhedra in $P$, respectively.  Here is a first consequence\footnote{Cf.\ \cite[Proposition 2.3.6(d)]{Maunder80}.} of Lemma \ref{l:tl}. 

\begin{cor}\label{c:lattice}For any polyhedron $P\seq\R^{n}$, both $\Subc{P}$ and $\Subo{P}$ are distributive lattices \textup{(}under set-theoretic intersections and unions\textup{)} bounded above by $P$ and below by $\emptyset$.\qed
\end{cor}
\begin{proof}Given polyhedra $A,B\seq P$, by Lemma \ref{l:tl} there is a triangulation $\Sigma$ of $P$ along with two subtriangulations $\Sigma_{A},\Sigma_{B}$ with $A=|\Sigma_{A}|$ and $B=|\Sigma_{B}|$. Then the triangulation $\Delta\coloneqq\Sigma_{A}\cap\Sigma_{B}$ triangulates $A\cap B$. Indeed, obviously $|\Delta|\seq A\cap B$. Conversely, if $x\in A\cap B$ then there are $\sigma_{A} \in \Sigma_{A}, \sigma_{B}\in\Sigma_B$ with $x\in\sigma_{A}$ and $x\in\sigma_{B}$. Setting $\tau\coloneqq\sigma_{A}\cap\sigma_{B}$, we have $x\in\tau\neq\emptyset$ and $\tau\in\Sigma$. Since $\tau$ is a face of $\sigma_{A}\in \Sigma_{A}$, $\tau \in \Sigma_{A}.$ Similarly, $\tau\in\Sigma_{B}.$ Hence $\tau\in\Delta$, and $A\cap B\seq |\Delta|$.  Similarly, it is elementary that the triangulation $\nabla\coloneqq\Sigma_{A}\cup\Sigma_{B}$ triangulates $A\cup B$. It is obvious that $P$ and $\emptyset$ are the upper and lower bounds of $\Subc{P}$. The statements about $\Subo{P}$ follow at once by taking complements.
\end{proof}
In Subsection \ref{ss:sub} we shall strengthen Corollary \ref{c:lattice} to the effect that  $\Subo{P}$ is a Heyting subalgebra of the Heyting algebra $\Op{(P)}$.
\subsection{Polyhedra:  dimension theory}\label{ss:polydim}The \emph{\textup{(}affine\textup{)} dimension} of a $d$-simplex  $\sigma=x_0\cdots x_d$ in $\R^{n}$ is the linear-space dimension of the affine subspace of $\R^{n}$ spanned by $\sigma$, and that dimension is precisely $d$ because of the affine independence of the vertices of $\sigma$.
The \emph{\textup{(}affine\textup{)} dimension} of a nonempty polyhedron $P$ in $\R^{n}$ is the maximum of the dimensions of all simplices contained in $P$; if $P=\emptyset$, its  dimension is $-1$. We write $\dim{P}$ for the dimension of $P$. Given a triangulation $\Sigma$ in $\R^{n}$, the \emph{\textup{(}combinatorial\textup{)} dimension} of $\Sigma$ is 
\[
\dim{\Sigma}\df \max{\left\{d\in \N\mid \text{ there exists } \sigma \in \Sigma \text{ such that $\sigma$ is a $d$-simplex.}\right\}}
\]
Again, the dimension of an empty triangulation is $-1$. Everything in the lemma that follows is of course classical.
\begin{lem}\label{l:td}For any polyhedron $\emptyset\neq P\seq\R^{n}$ and every $d\in\N$, the following are equivalent.
\begin{enumerate}[label={\rm (\roman*)}]
\item $\dim{P}=d$.
\item There exists a triangulation $\Sigma$ of $P$ such that $\dim{\Sigma}=d$.
\item All triangulations $\Sigma$ of $P$ satisfy $\dim{\Sigma}=d$. 
\item The Lebesgue covering dimension \cite[Definition 3.1.1]{Pears75} of the topological space $P$ is $d$.
\end{enumerate}
\end{lem}
\begin{proof}With the Triangulation Lemma \ref{l:tl} available, the equivalences (i) $\Leftrightarrow$ (ii) $\Leftrightarrow$ (iii) follow from linear algebra. The equivalence  (i) $\Leftrightarrow$ (iv) is, in essence, the Lebesgue Covering Theorem \cite[Theorem IV 2]{HurewiczWallman41}.
\end{proof}
\section{The locally finite Heyting algebra of a polyhedron}\label{s:onepoly}
Throughout this section we fix $n\in\N$ along with a polyhedron $P\seq\R^{n}$. We shall study the distributive lattice $\Subo{P}$ (Corollary \ref{c:lattice}). We begin by proving that $\Subo{P}$ is in fact a Heyting algebra. We then prove that $\Subo{P}$ is always locally finite.
\subsection{The Heyting algebra of open subpolyhedra}\label{ss:sub}
 Let us record a well-known, elementary observation on relative interiors for which we know no convenient reference.
\begin{lem}\label{l:ri}
Let $\Sigma$ be a triangulation in $\R^n$, let $\tau=x_0\cdots x_d$ be a simplex of $\Sigma$, and let $x\in\rint{\tau}$. Then  no proper face $\sigma \faceneq \tau$ contains $x$. Hence, in particular, the carrier $\sigma^{x}$ of $x$ in $\Sigma$ is the inclusion-smallest simplex of $\Sigma$ containing $x$.
\end{lem}
\begin{proof}There are $r_0,\dots,r_d\in(0,1]$ such that $x=\sum_{i=0}^dr_i x_i$ and $\sum_{i=0}^dr_i=1$. Let $\rho_i \df x_0 \cdots x_{i-1} x_{i+1} \cdots x_d$. Clearly $\rho_i \faceneq \tau$ for each $i\in\{0,\ldots,d\}$, and for each $\sigma \faceneq \tau$ there exists $i\in\{0,\ldots,d\}$ such that $\sigma \face \rho_i$. Hence, if we assume by way of contradiction that $x\in\sigma \faceneq \tau$, then $x\in\rho_i$ for some  $i\in\{0,\dots,d\}$; say $x\in\rho_0$. Then $x=\sum_{i=1}^ds_i x_i$, for some $s_1,\dots,s_d\in[0,1]$ such that $\sum_{i=1}^ds_i=1$. It follows that $r_0=\sum_{i=1}^d(s_i-r_i)$, and so
 \begin{equation*}
   0 = x - x = \sum_{i=1}^ds_i x_i - \sum_{i=0}^dr_i x_i = \sum_{i=1}^d(s_i - r_i) x_i - r_0x_0= \sum_{i=1}^d(s_i - r_i) (x_i-x_0).
 \end{equation*}
 Since $r_0>0$, there must be  $i\in\{1,\dots,d\}$ such that $s_i - r_i\neq0$, contradicting  the affine independence of  $x_0,\dots,x_d$.
\end{proof}
The next lemma is the key fact of this subsection.\footnote{Cf.\ \cite[Proposition 2.3.7]{Maunder80}.}
\begin{lem}
 \label{l:Hneg}
Let $P$ and $Q$ be  polyhedra in $\R^n$ with $Q \subseteq P$, and suppose $\Sigma$ is a triangulation of $P$ such that
\[
\Sigma_{Q}\df\left\{\sigma\in\Sigma\mid \sigma\seq Q\right\}
\]
triangulates $Q$. Define
\begin{itemize}
\item $C\df\cl{(P\setminus Q)}$,
\item $\Sigma_{C} \df \left\{\sigma\in \Sigma \mid \sigma \subseteq C\right\}$, and
\item $\Sigma^*\df \left\{\sigma\in \Sigma \mid \text{There exists } \tau \in \Sigma\setminus \Sigma_Q \text{ such that } \sigma\face\tau\right\}$.

\end{itemize}
 Then 
 \begin{enumerate}[label={\textup{(\arabic*)}}]
 \item\label{l:Hneg1} $\Sigma_{C}=\Sigma^{*}$, and 
 \item\label{l:Hneg2} $|\Sigma_{C}|=|\Sigma^{*}|=C$.
 \end{enumerate}
In particular, $C$ is a polyhedron.
 \end{lem}
\begin{proof}
We first show  that $\Sigma^*$ triangulates $C$, that is: 
 \begin{equation*}\tag{*}\label{eq:star}
  |\Sigma^*| \df \bigcup \Sigma^* = C.
 \end{equation*}
 To show $|\Sigma^*| \seq C$, let $\sigma\in \Sigma^*$, and pick $\tau\in \Sigma \setminus \Sigma_Q$  such that $\sigma \face \tau$. We prove that $\rint{\tau} \seq P \setminus Q$. For, if  $x\in\rint{\tau}$, by Lemma \ref{l:ri} there are no simplices $\sigma\in \Sigma$ such that $x\in \sigma \faceneq \tau$. Then, by definition of triangulation, for any simplex $\rho\in \Sigma$, $x\in\rho$ entails  $\tau\face\rho$. Hence no simplex of $\Sigma_Q$ contains $x$, or equivalently, $x\not\in Q$ and therefore $\rint{\tau}\subseteq P \setminus Q$. 
 
Now, it is clear that any simplex $\tau$ satisfies $\tau=\cl{\rint{\tau}}$. It follows that  $\sigma \seq \tau = \cl{\rint{\tau}} \subseteq \cl{(P \setminus Q)}$, and thus $|\Sigma^*| \subseteq C$ as was to be shown.
 
 \smallskip
Conversely, to show $C \seq |\Sigma^*|$, let $x \in C$. Since $C$ is the closure of $P\setminus Q$ in $\R^n$,  there exists a sequence $\{x_i\}_{i\in\N}\seq P \setminus Q$ that converges to $x$. Clearly the carrier $\sigma^{x_{i}}$ of $x_{i}$ in $\Sigma$ lies in $\Sigma\setminus \Sigma_Q$, for all $i\in\N$. Since $\Sigma \setminus \Sigma_Q$ is finite, there must exist a simplex $\tau \in \Sigma \setminus \Sigma_Q$ containing infinitely many elements of  $\{x_i\}_{i\in\N}$. Then there exists a subsequence  of $\{x_i\}_{i\in\N}$that is contained in $\tau$ and converges to $x$. Since $\tau$ is closed, $x\in\tau$, and therefore $x\in |\Sigma^*|$ as was to be shown.

This establishes \eqref{eq:star}. It now suffices to prove \ref{l:Hneg1}. For the non-trivial inclusion  $\Sigma_{C}\seq\Sigma^{*}$, let $\sigma\in\Sigma$ be such that $\sigma \seq C$, and pick $\beta\in\rint{\sigma}$.  
There is a sequence $\{x_i\}_{i\in\N}\seq P\setminus Q$ converging to $\beta\in\sigma$.  Since each $x_i$ is in some simplex of $\Sigma \setminus \Sigma_Q$ and $\Sigma$ is finite, there must exist a simplex $\tau \in \Sigma\setminus \Sigma_Q$  containing a subsequence of $\{x_i\}_{i\in\N}$ that converges to $\beta$. Since $\tau$ is closed, $\beta\in\tau$. But by Lemma \ref{l:ri}, $\sigma^{\beta}=\sigma$, so that $\sigma \seq \tau$ and $\sigma\in\Sigma^{*}$.
\end{proof}
\begin{cor}
 \label{c:clpol}
 Given polyhedra $Q_1,Q_2$ in $\R^n$, the set $\cl{(Q_2 \setminus Q_1)}$ is a polyhedron.
 \begin{proof}
  Observe that $Q_2 \backslash Q_1 = Q_2 \backslash (Q_1 \cap Q_2)$ and apply Corollary \ref{c:lattice} together with Lemma \ref{l:Hneg} to the set $P\coloneqq \conv{(Q_1 \cup Q_2)}$, which clearly is a polyhedron.
 \end{proof}
\end{cor}
\begin{cor}\label{c:heytstruct}
 The lattice $\Subc{P}$ of is closed under the co-Heyting implication \eqref{e:ci} of $\K{(P)}$. Dually, the lattice $\Subo{P}$ is closed under the Heyting implication \eqref{e:hi} of $\Op{(P)}$.
\end{cor}
\begin{proof}The first statement is an immediate consequence of Corollary \ref{c:clpol}. The second statement follows by dualising.
\end{proof}
\subsection{Local finiteness through triangulations}\label{ss:lf} Having established that $\Subo{P}$ is a Heyting subalgebra of $\Op{(X)}$, we infer an important structural property of $\Subo{P}$, local finiteness. For this, we first identify the class of subalgebras of $\Subo{P}$ that corresponds to triangulations of $P$. These algebras will have a central r\^{o}le in the sequel, too.
\begin{defn}[$\Sigma$-definable polyhedra]\label{d:hasigma}For any triangulation $\Sigma$ in $\R^{n}$, we write
$\PCc{(\Sigma)}$
for the sublattice of $\K{(|\Sigma|)}$ generated by $\Sigma$, and $\PCo{(\Sigma)}$ for the sublattice of $\Op{(|\Sigma|)}$ generated by $\{|\Sigma|\setminus C\mid C \in \PCc{(\Sigma)} \}$. We call $\PCc{(\Sigma)}$ the set of \emph{$\Sigma$-definable polyedra}, and  $\PCo{(\Sigma)}$ the set of \emph{$\Sigma$-definable open polyedra}. 
\end{defn}
Note that we have
\[
\PCc{(\Sigma)}=\{C\seq \R^{n}\mid C \text{ is the union of some subset of } \Sigma\}.
\]
\begin{lem}\label{l:triangheyt}For any triangulation $\Sigma$ of $P$, $\PCc{(\Sigma)}$ is a co-Heyting subalgebra of $\Subc{P}$. Dually,  $\PCo{(\Sigma)}$ is a Heyting subalgebra of $\Subo{P}$.
\end{lem}
\begin{proof}For any $\emptyset\neq C,D\in\PCc{(\Sigma)}$, it follows immediately by the assumptions that  $C$ and $D$ are triangulated by the collection of simplices of $\Sigma$ contained in $C$ and $D$, respectively. Hence $C\leftarrow D\df\cl{(C\setminus D)}=|\Sigma^{*}|=\bigcup\Sigma^{*}$ by Corollary \ref{c:clpol} and Lemma \ref{l:Hneg}, where $\Sigma^{*}$ is the appropriate subset of $\Sigma$ as per Lemma \ref{l:Hneg}. Thus $C\leftarrow D\in\PCc{(\Sigma)}$.
\end{proof}

\begin{cor}\label{c:lf1}Let $H$ be the co-Heyting subalgebra of $\Subc{P}$ generated by finitely many polyhedra $P_{1},\ldots, P_{m}\seq  P$. Let further $\Sigma$ be any triangulation of $P$ that triangulates each $P_{i}$, $i\in\{1,\ldots,m\}$. Then $H$ is a co-Heyting subalgebra of $\PCc{(\Sigma)}$. In particular, $H$ is finite. Dually for the Heyting subalgebra of $\Subo{P}$ generated by $P\setminus P_{i}$, $i\in\{1,\ldots,m\}$.
\end{cor}
\begin{proof} Each $P_{i}$ is the union of those simplices of $\Sigma$ that are contained in $P_{i}$, by assumption. It follows  that  the distributive lattice $L$ generated in $\Subc{P}$  by $\{P_{1},\ldots,P_{m}\}$ is entirely contained in $\PCc{(\Sigma)}$. Now, if $C,D\in L$, $C\leftarrow D\df\cl{(C\setminus D)}=|\Sigma^{*}|=\bigcup\Sigma^{*}$ by Corollary \ref{c:clpol} and Lemma \ref{l:Hneg}, where $\Sigma^{*}$ is the appropriate subset of $\Sigma$ as per Lemma \ref{l:Hneg}. Hence $C\leftarrow D\in\PCc{(\Sigma)}$, as was to be shown.
\end{proof}
\begin{cor}\label{c:lf2}The Heyting algebra $\Subo{P}$ is locally finite, and so is the co-Heyting algebra $\Subc{P}$.
\end{cor}
\begin{proof}The second statement is Corollary \ref{c:lf1} together with the Triangulation Lemma \ref{l:tl}. The first statement follows by dualising.
\end{proof}
\section{Topological dimension and bounded depth}\label{s:dim}
The aim of this section is to prove:
\begin{thm}\label{t:dim}For any polyhedron $\emptyset\neq P\seq\R^{n}$ and every $d\in\N$, the following are equivalent.
\begin{enumerate}[label={\rm (\roman*)}]
\item $\dim{P}=d$.
\item The Heyting algebra $\Subo{P}$ satisfies  the equation $\BD_{d}=\top$, and fails each equation $\BD_{d'}=\top$ for each integer $0\leq d'<d$.
\end{enumerate}
\end{thm}
We deduce the theorem from a combinatorial counterpart of the result for triangulations, Lemma \ref{l:ji} below. In turn, this lemma will follow from the analysis of frames arising from triangulations that we carry out first.
\subsection{Frames of algebras of definable polyhedra.}\label{ss:frames} Consider a triangulation $\Sigma$, and the finite Heyting algebra $\PCo{(\Sigma)}$. We shall henceforth  regard $\Sigma$ as a poset under the inclusion order, whenever convenient. Note that the inclusion order of $\Sigma$ is the same thing as the ``face order" $\sigma\face\tau$ we have been using above: since $\Sigma$ is a triangulation (as opposed to a mere set of simplices),  $\sigma\seq\tau$ implies $\sigma\face\tau$, and the converse implication is obvious. Indeed, the poset $\Sigma$ is a much-studied object in combinatorics, where it is known as the \emph{face poset} of a simplicial complex. We next show what r\^ole $\Sigma$ plays for the Heyting algebra  $\PCo{(\Sigma)}$, by establishing an isomorphism of Heyting algebras $\Up{\Sigma}\cong\PCo{(\Sigma)}$; equivalently, through  Esakia duality (Lemma \ref{l:esakia}), the face poset $\Sigma$ is isomorphic to the dual  frame of the algebra 
 $\PCo{(\Sigma)}$. This result is technically important, because the prime filters of  $\PCo{(\Sigma)}$, or what amounts to the same, its join-irreducible elements, are  somewhat harder to visualise than the simplices of $\Sigma$.  There are  corresponding results for the co-Heyting algebra $\PCc{(\Sigma)}$ which we do not spell out as we do not need them for the proof of our main result.

We recall the notion of open star of a simplex, cf.\ e.g.\ \cite[Definition 2.4.2]{Maunder80}.
\begin{defn}[Open star]\label{d:ostar}For $\Sigma$ a triangulation,
 the \emph{open star} of   $\sigma\in\Sigma$ is the subset of $|\Sigma|$ defined by
 \begin{equation*}
  \ostar{\sigma} \coloneqq \bigcup_{\sigma\seq\tau\in\Sigma} \rint{\tau}.
 \end{equation*} 
\end{defn}
Although not immediately obvious, it is classical (see e.g.\ \cite[Proposition 2.4.3]{Maunder80}) that the open star of any simplex is an open subpolyhedron, that is, for each $\sigma\in\Sigma$
\begin{align}\label{eq:starisopen}
\ostar{\sigma} \in \PCo{(\Sigma)}.
\end{align}
Indeed, set
\[
K_\sigma\coloneqq\{\tau\in\Sigma\mid \sigma \not \seq \tau\}.
\]
Then $K_\sigma$ is clearly a subtriangulation of $\Sigma$,  $|K_\sigma|$ is a subpolyhedron of $|\Sigma|$, and thus $O\coloneqq|\Sigma|\setminus |K_\sigma| \in \PCo{(|\Sigma|)}$; but one can show using Lemma \ref{l:interior} that $O=\ostar{\sigma}$, so  \eqref{eq:starisopen} holds. 

We now define a function
\begin{align}
\gammaup \colon \Up{\Sigma} &\longrightarrow\PCo{(\Sigma)}\label{eq:gup}\\
U\in\Up{\Sigma} &\longmapsto \bigcup_{\sigma\in U}\rint{\sigma}.\nonumber
\end{align}
To see that $\gammaup$ is well-defined, use the fact that $\Sigma$ is a finite poset to list the minimal elements $\sigma_1,\ldots,\sigma_u$ of the upper set  $U$. Then
\[
U=\uparrow{}\sigma_1\cup\cdots \cup\uparrow{}\sigma_u,
\]
so that 
\begin{align}
\gammaup(U)&=\gammaup(\uparrow{}\sigma_1)\cup\cdots \cup \gammaup(\uparrow{}\sigma_u) \nonumber\\
&=\left(\bigcup_{\sigma_1\subseteq\tau\in\Sigma}\rint{\tau}\right)\cup\cdots \cup \left(\bigcup_{\sigma_u\subseteq\tau\in\Sigma}\rint{\tau}\right) \nonumber\\
&=\ostar{\sigma_1}\cup\cdots\cup\ostar{\sigma_u}.\nonumber
\end{align}
Thus $\gammaup(U)$ is a union of open stars and hence a member of $\PCo{(\Sigma)}$.

\begin{lem}\label{l:updown}The map $\gammaup$ of  \eqref{eq:gup} is an   isomorphism of the finite Heyting algebras $\Up{\Sigma}$ and $\PCo{(\Sigma)}$.
\end{lem}
\begin{proof} It suffices to show that $\gammaup$ is an isomorphism of distributive lattices. It is clear that $\gammaup$ preserves the top and bottom elements, and that it preserves unions: if $U,V\in \Up{\Sigma}$ then
\[
\gammaup(U\cup V)=\bigcup_{\sigma\in U\cup V}\rint{\sigma}=\left(\bigcup_{\sigma\in U }\rint{\sigma}\right)\cup\left(\bigcup_{\sigma\in  V}\rint{\sigma}\right)=\gammaup(U)\cup\gammaup(V).
\]
Concerning intersections,
\begin{align}
\gammaup(U)\cap\gammaup(V)&=\left(\bigcup_{\sigma\in U}\rint{\sigma}\right)\cap\left(\bigcup_{\tau\in V }\rint{\tau}\right)\nonumber\\
&=\bigcup_{\sigma\in U}\left(\rint{\sigma}\cap\bigcup_{\tau\in V }\rint{\tau}\right)\nonumber\\
&=\bigcup_{\sigma\in U}\bigcup_{\tau\in V}\left(\rint{\sigma}\cap\rint{\tau}\right)\nonumber\\
&=\bigcup_{\sigma\in U,\, \tau \in V}\left(\rint{\sigma}\cap\rint{\tau}\right)\label{eq:intersection}
\end{align}
By Lemma \ref{l:interior}, for any two $\sigma,\tau\in\Sigma$ the intersection $\rint{\sigma}\cap\rint{\tau}$ is empty as soon as $\sigma\neq\tau$. Hence from \eqref{eq:intersection} we deduce
\[
\gammaup(U)\cap\gammaup(V)=\bigcup_{\delta\in U\cap V}\rint{\delta}=\gammaup(U\cap V),
\]
as was to be shown.

To prove $\gammaup$ is surjective, let $O\in\PCo{(\Sigma)}$ and set $P\coloneqq|\Sigma|\setminus O\in\PCc{(\Sigma)}$. Then, by definition of $\PCc{(\Sigma)}$, there is exactly one subtriangulation $\Delta$ of $\Sigma$ such that $P=|\Delta|$, and $\Delta$ is a lower set of (the poset) $\Sigma$. Set
$U\coloneqq\Sigma\setminus \Delta$, so that $U$ is an upper set of $\Sigma$. We show:
\begin{align}\label{eq:surj}
O=\bigcup_{\sigma\in U}\rint{\sigma}.
\end{align}
To prove \eqref{eq:surj} we use the fact that, since $P$ is a member of $\PCc{(\Sigma)}$, for every $\sigma\in \Sigma$ we have 
\begin{align}\label{eq:surj2}
\rint{\sigma}\cap P\neq\emptyset \  \text{if, and only if,} \ \sigma\seq P.
\end{align}
Only the left-to-right implication in \eqref{eq:surj2} is non-trivial, and we prove the contrapositive. Assume $\sigma\not\seq P$. If $\sigma\cap P=\emptyset$ obviously $\rint{\sigma}\cap P=\emptyset$. Otherwise $\tau\coloneqq\sigma\cap P$ must be a \emph{proper} face of $\sigma$, and therefore $\rint{\sigma}\cap \tau=\emptyset$; hence $\rint{\sigma}\cap P=\emptyset$. This establishes  \eqref{eq:surj2}.

Now, to show \eqref{eq:surj},  if $x\in O$ then the carrier $\sigma^{x}\in\Sigma$ is such that $\rint{\sigma^{x}}\cap P=\emptyset$, so $\sigma^{x}\not\seq P$; equivalently,  $\sigma^{x}\not \in \Delta$. Then $\sigma^{x}\in U$ and hence $x \in \bigcup_{\sigma\in U}\rint{\sigma}$. Conversely, if $x \not \in O$, then $x \in P$, so $\rint{\sigma^{x}}\cap P \neq\emptyset$ and thus $\sigma^{x}\seq P$; equivalently,   $\sigma^{x}\in\Delta$ . Then  $\sigma^{x}\not\in U$ and hence $x\not\in \bigcup_{\sigma\in U}\rint{\sigma}$. This proves  \eqref{eq:surj}.

In light of \eqref{eq:surj} we now have $\gammaup(U)=O$, so that $\gammaup$ is surjective.

Finally, to prove injectivity, it suffices to recall that relative interiors of simplices in $\Sigma$ are pairwise-disjoint, so the union in \eqref{eq:gup} is in fact a disjoint one, which makes the injectivity of $\gammaup$ evident.
\end{proof}
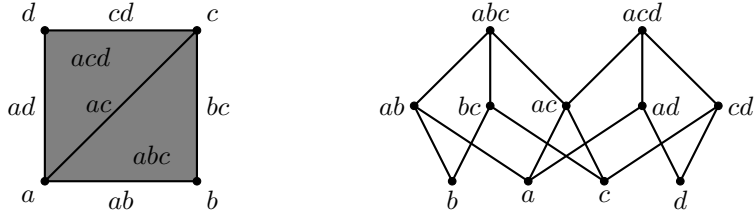
\begin{figure}
\begin{tikzpicture}
\path [fill, gray] (0,0) -- (2,0) -- (2,2) -- (0,0);
\path [fill, gray] (0,0) -- (0,2) -- (2,2) -- (0,0);
\draw [thick] (0,2) -- (0,0) -- (2,0) -- (2,2) -- (0,2);
\draw [thick] (0,0) -- (2,2);
\draw [fill] (0,0) circle [radius=0.05];
\draw [fill] (0,2) circle [radius=0.05];
\draw [fill] (2,0) circle [radius=0.05];
\draw [fill] (2,2) circle [radius=0.05];
\node [below left] at (0,0) {$a$};
\node [above left] at (0,2) {$d$};
\node [below right] at (2,0) {$b$};
\node [above right] at (2,2) {$c$};
\node [below] at (1,0) {$ab$};
\node [above] at (1,2) {$cd$};
\node [right] at (2,1) {$bc$};
\node [left] at (0,1) {$ad$};
\node [left] at (1,1) {$ac$};
\node [below] at (1.4,0.6) {$abc$};
\node [above] at (0.6,1.4) {$acd$};
\end{tikzpicture}
\hspace{1.5cm}
\begin{tikzpicture}
\draw [thick] (7,2) -- (6,1);
\draw [thick] (7,2) -- (7,1);
\draw [thick] (7,2) -- (8,1);
\draw [thick] (9,2) -- (8,1);
\draw [thick] (9,2) -- (9,1);
\draw [thick] (9,2) -- (10,1);
\draw [thick] (6.5,0) -- (6,1);
\draw [thick] (6.5,0) -- (7,1);
\draw [thick] (7.5,0) -- (6,1);
\draw [thick] (7.5,0) -- (8,1);
\draw [thick] (7.5,0) -- (9,1);
\draw [thick] (8.5,0) -- (7,1);
\draw [thick] (8.5,0) -- (8,1);
\draw [thick] (8.5,0) -- (10,1);
\draw [thick] (9.5,0) -- (9,1);
\draw [thick] (9.5,0) -- (10,1);
\draw [fill] (7,2) circle [radius=0.05];
\draw [fill] (9,2) circle [radius=0.05];
\draw [fill] (6,1) circle [radius=0.05];
\draw [fill] (7,1) circle [radius=0.05];
\draw [fill] (8,1) circle [radius=0.05];
\draw [fill] (9,1) circle [radius=0.05];
\draw [fill] (10,1) circle [radius=0.05];
\draw [fill] (6.5,0) circle [radius=0.05];
\draw [fill] (7.5,0) circle [radius=0.05];
\draw [fill] (8.5,0) circle [radius=0.05];
\draw [fill] (9.5,0) circle [radius=0.05];
\node [below] at (7.5,0) {$a$};
\node [below] at (9.5,0) {$d$};
\node [below] at (6.5,0) {$b$};
\node [below] at (8.5,0) {$c$};
\node [left] at (6,1) {$ab$};
\node [right] at (10,1) {$cd$};
\node [left] at (7,1) {$bc$};
\node [right] at (9,1) {$ad$};
\node [left] at (8,1) {$ac$};
\node [above] at (7,2) {$abc$};
\node [above] at (9,2) {$acd$};
\end{tikzpicture}
\protect{\caption{A triangulation $\Sigma$ of $[0,1]^{2}$ and the corresponding intuitionistic frame that is (isomorphic to) the Esakia-dual of the Heyting algebra $\PCo{(\Sigma)}$ of  $\Sigma$-definable open polyhedra. Cf.\ Example \ref{ex:square}.}\label{fig1}}
\end{figure}
\begin{ex}\label{ex:square}Consider the unit square $[0,1]^{2}$, and let    $\Sigma$ be its triangulation shown on the left of Fig.\ \ref{fig1}. The reader can verify  that the set $\Sigma$ ordered by inclusion --- whose Hasse diagram is depicted on the right in Fig.\ \ref{fig1} --- is isomorphic to the  Esakia-dual poset of the Heyting algebra $\PCo{(\Sigma)}$ of  $\Sigma$-definable open polyhedra.\qed
\end{ex}
\subsection{Topological dimension through bounded depth}\label{ss:tdbd}
We can now prove:
\begin{lem}\label{l:ji}Let $\Sigma$ be a  triangulation in $\R^{n}$. 
\begin{enumerate}[label={\rm (\arabic*)}]
\item\label{l:ji1} The join-irreducible elements of $\PCc{(\Sigma)}$ are the simplices of $\Sigma$.
\item\label{l:ji2} The join-irreducible elements of $\PCo{(\Sigma)}$ are the open stars of simplices of $\Sigma$.
\item\label{l:ji3} In both $\PCc{(\Sigma)}$ and $\PCo{(\Sigma)}$ there is a chain of prime filters having cardinality $\dim{\Sigma}+1$. In neither $\PCc{(\Sigma)}$ nor $\PCo{(\Sigma)}$  is there a chain of  prime filters having strictly larger cardinality.
\end{enumerate}
\end{lem}
\begin{proof}Item \ref{l:ji1} follows from direct inspection of the definitions. Item \ref{l:ji2} is an immediate consequence of Lemma \ref{l:updown} along with Esakia duality (Subsection \ref{ss:esakia}). To prove \ref{l:ji3},
set $d\df\dim \Sigma$ and note that by definition $\Sigma$ contains at least one $d$-simplex $\sigma=x_{0}\cdots x_{d}\in\Sigma$. By item \ref{l:ji1} the chain of simplices
$x_0<x_0x_{1}<\cdots<x_{0}x_{1}\cdots x_{d}=\sigma$ is a chain of join-irreducible elements of $\PCc{(\Sigma)}$, and  the principal filters generated by these elements yields a chain of prime filters of $\PCc{(\Sigma)}$ of cardinality $d+1$. On the other hand, any chain of prime filters of $\PCc{(\Sigma)}$ must be finite because $\PCc{(\Sigma)}$ is. If $\mathfrak{p}_{1}\subset\mathfrak{p}_2\subset\cdots\subset\mathfrak{p}_{l}$ is any such chain of prime filters, then each $\mathfrak{p}_{i}$ is principal --- again because $\PCc{(\Sigma)}$ is finite ---  its unique generator $p_{i}$ is join-irreducible, and we have $p_{l}<p_{l-1}<\cdots <p_{2}<p_1$ in the order of the lattice $\PCc{(\Sigma)}$. Then $p_{i}\in\Sigma$, and clearly,  since the simplex $p_{1}$ has $l-1$  proper faces of distinct dimensions, $\dim{p_{1}}\geq l-1$. But  $d\geq \dim{p_{1}}$ by definition of $d\df\dim{\Sigma}$, and therefore $d+1\geq l$, as was to be shown. The proof for $\PCo{(\Sigma)}$ is analogous, using item \ref{l:ji2}.
\end{proof}
To finally relate the  bounded-depth formul\ae\ to topological dimension, we give a  proof of Theorem  \ref{t:dim}.
\begin{proof}[Proof of Theorem \ref{t:dim}](i) $\Rightarrow$ (ii) \ By Lemma  \ref{l:td}, $\dim{\Sigma}=d$ for any triangulation $\Sigma$ of $d$. By Lemmas \ref{l:bd}, \ref{l:triangheyt}, and \ref{l:ji}, the subalgebra $\PCo{(\Sigma)}$ of $\Subo{P}$ satisfies the equation $\BD_{d}=\top$, and fails each equation $\BD_{d'}=\top$ for each integer $0\leq d'<d$. To complete the proof it thus suffices to show that any finitely generated subalgebra of $\Subo{P}$ is a subalgebra of  $\PCo{(\Sigma)}$ for some triangulation $\Sigma$ of $P$. But this is precisely the content of the Triangulation Lemma \ref{l:tl}.

\smallskip \noindent (ii) $\Rightarrow$ (i) \ We prove the contrapositive. Suppose first $\dim{P}>d\geq 0$. Then, by (i) $\Rightarrow$ (ii), $\Subo{P}$ fails the equation $\BD_{d}$, so that (ii) does not hold. On the other hand, if $0\leq d'\df\dim{P}<d$, by (i) $\Rightarrow$ (ii) we know that $\Subo{P}$ satisfies the equation $\BD_{d'}=\top$, so again (ii) does not hold.
\end{proof}

\section{Nerves of posets, and the geometric finite model property}\label{s:nerves}
In this section we use a classical construction in polyhedral geometry to realise finite posets geometrically.  Our aim is to prove:
\begin{thm}\label{t:nerve}Let $A$ be a finite, nonempty poset of cardinality $n\in\N$. There exists a triangulation $\Sigma$ in $\R^{n}$ satisfying the following conditions.
\begin{enumerate}[label={\rm (\arabic*)}]
\item\label{t:nerve1} $\depth{A}=\dim{\Sigma}$.
\item\label{t:nerve2} There is a surjective p-morphism $\Sigma\twoheadrightarrow A$, where   $\Sigma$ is equipped with the inclusion order.
\end{enumerate}
\end{thm}
\begin{construction}
 The \emph{nerve} (\cite[{\textit{passim}}]{Alexandrov98}, \cite[p.\ 1844]{Bjorner95}) of a finite poset $A$ is the set
\[
\nerve{(A)}\df\left\{\emptyset\neq C\seq A \mid C \text{ is totally ordered by the restriction of $\leq$ to } C\times C \right\}.
\] 
In other words, the nerve of $A$ is the collection of all chains of $A$. We always regard the nerve $\nerve{(A)}$ as a poset under inclusion order.\footnote{In the literature on polyhedral geometry the nerve is most often regarded as an ``abstract simplicial complex'', or ``vertex scheme''. See e.g.\ \cite{Alexandrov98}. We do not need to explicitly use this notion in this paper.} Let us display the elements of $A$ as $\left\{a_{1},\ldots,a_{n}\right\}$. Let $e_{1},\ldots,e_{n}$ denote the vectors in the standard basis of the linear space $\R^{n}$. The \emph{triangulation induced by} the nerve $\nerve{(A)}$ is the set of simplices
\
\[
\geo{(\nerve{(A)})}\df\left\{\conv{\left\{e_{i_1},\ldots,e_{i_l}\right\}}\seq\R^{n}\mid \left\{a_{i_1},\ldots,a_{i_l}\right\}\in\nerve{(A)}\right\}.
\] 
Then it is immediate that $\geo{(\nerve{(A)})}$ indeed is a triangulation in $\R^{n}$, and its underlying polyhedron $|\geo{(\nerve{(A)})}|$ is called the \emph{geometric realisation of} the poset $A$. For the proof of Theorem \ref{t:nerve}, we set
\[
\Sigma\df \geo{(\nerve{(A)})}.
\]
Using the fact that simplices are uniquely determined by their vertices (see Subsection \ref{ss:polybasic}), we see that the map
\[
a_{i_1}<a_{i_{2}}<\cdots< a_{i_{l}}\in \nerve{(A)}\ \ \longmapsto \ \ \conv{\left\{e_{i_1},\ldots,e_{i_l}\right\}} \in \Sigma
\]
is an order-isomorphism between $\nerve{(A)}$ and $\Sigma$, the latter ordered by inclusion.
Therefore,
\[
\dim{\Sigma}=\text{cardinality of the longest chain in $A$}=\depth{A},
\]
so that \ref{t:nerve1} holds.  To prove Theorem \ref{t:nerve} it will therefore suffice to construct a p-morphism $\nerve{(A)}\twoheadrightarrow A$. To this end, let us  define a function 
\begin{align*}
f\colon \nerve{(A)}&\longrightarrow A\\
C\in \nerve{(A)}&\longmapsto \max{C}\in A,
\end{align*}
where the maximum is computed in the poset $A$.
\qed
\end{construction}
\begin{proof}[Proof of Theorem \ref{t:nerve}]To show that $f$ preserves order, just note that $C\seq D \in \nerve{(A)}$ obviously entails $\max{C}\leq\max{D}$ in $A$. To show that $f$ is a p-morphism, for each $C\in \nerve{(A)}$ we prove:
\begin{align}\label{eq:tbp}
f[\uparrow{}C]=\left\{a_{k}\in A \mid a_{k}\geq \max{C}\right\}\eqqcolon\,\uparrow{}\max{C}\eqqcolon\,\uparrow{}f(C).
\end{align}
Only the first equality in \eqref{eq:tbp} needs proof, and only the   right-to-left inclusion is non-trivial. So let $a_{k}\in A$ be such that $a_{k}\geq \max{C}$. Then the set $D\df C\cup\{a_{k}\}$ is a chain in $A$, i.e.\ a member of $\nerve{(A)}$, and $D\in\, \uparrow{}C$ because $C\seq D$. Further, 
$\max{D}=a_{k}$, because $a_{k}\geq \max{C}$, so that $f(D)=a_{k}$. Hence $a_{k}\in f[\uparrow{}C]$, and the proof is complete.
\end{proof}
\begin{rem}\label{l:universal}The reader may be interested in comparing the construction above of the Heyting  algebra  $\Up{\Sigma}$ from the finite distributive lattice $\Up{A}$ with the description of the  prelinear Heyting algebra\footnote{A Heyting algebra is \emph{prelinear} if it satisfies the  law $(x\to y)\vee (y\to x)=\top$.}  freely generated by a finite distributive lattice in \cite{AGM}, along with that of the Heyting algebra freely generated by a finite distributive lattice in \cite{Ghilardi92} (see also \cite{GZ}). It is an interesting open question whether the construction given here using the nerve is the solution to a universal problem, too.\qed
\end{rem}
\section{Proof of Theorem}\label{s:main}
\begin{proof}[Proof of Theorem] By Theorem \ref{t:dim}, $\Log{\Polyd}$ is contained in  intuitionistic logic extended by the axiom schema $\BD_{d}$. Conversely, suppose a formula $\alpha$ is not contained in intuitionistic logic extended by the axiom schema $\BD_{d}$. By Lemmas \ref{l:birkhoff}, \ref{l:fmp} and \ref{l:bd}, there exists a finite poset $A$ satisfying $\depth{A}\leq d$ such that  there is an evaluation into the frame $A$ that provides a counter-model to $\alpha$; equivalently, the equation $\alpha=\top$ fails in the Heyting algebra $\Up{A}$.  By Theorem \ref{t:nerve} there exists a triangulation $\Sigma$ in $\R^{|A|}$ such that $\depth{A}=\dim{\Sigma}\leq d$, along with  a surjective p-morphism 
\begin{align}\label{eq:pmor}
p\colon\Sigma\longsurj A.
\end{align} We set $P\coloneqq |\Sigma|$ and consider the Heyting algebra 
$\Subo{P}$ and its subalgebra $\PCo{(\Sigma)}$, per Corollary \ref{c:heytstruct} and Lemma \ref{l:triangheyt}, respectively. Since $\dim{P}\leq d$, we have $\Log{P}\supseteq \Log{\Polyd}$ by Theorem \ref{t:dim}.

By Lemma \ref{l:updown} there is an isomorphism of (finite) Heyting algebras
\[
\gammaup\colon \Up{\Sigma}\longrightarrow \PCo{(\Sigma)}
\]
defined as in \eqref{eq:gup}. By finite Esakia duality (Lemma \ref{l:esakia}) we have  isomorphisms of posets
\[
\Sigma\cong\Spec{\Up{\Sigma}}\cong\Spec{\PCo{(\Sigma)}}.
\]
The  Esakia dual $\Spec{p}\colon \Up{A}\hookrightarrow \Up{\Sigma}$ of the surjective p-morphism \eqref{eq:pmor} is an injective homomorphism. We thus have homomorphisms
\[
\Up{A}\xhookrightarrow{\,\Spec{p}\,} \Up{\Sigma}\overset{\gammaup}\cong \PCo{(\Sigma)}\seq\Subo{P},
\]
where the  inclusion preserves the Heyting structure  by Lemma   \ref{l:triangheyt}. Since the equation $\alpha=\top$ fails in  $\Up{A}$, it also fails in the larger algebra $\Subo{P}$; equivalently, $\alpha\not \in \Log{P}\supseteq \Log{\Polyd}$, and the proof of the first statement is complete. The second statement follows easily from the first using Lemma \ref{l:fmp}.
\end{proof}
\begin{rem}\label{r:manifolds}Intuitionistic logic is capable of expressing properties of polyhedra other than their  dimension. To show  this, let $\P$ consist of the class of all polyhedra that are, as topological spaces, closed (=without boundary) topological manifolds. Then $\Log{\P}$ contains intuitionistic logic properly. Indeed, it is a classical theorem that for any triangulation $\Sigma$ of any $d$-dimensional manifold $M\in\P$, each  $(d-1)$-simplex $\sigma\in\Sigma$ is a face of exactly two $d$-simplices of $\Sigma$. It follows from our results above that  the class $\P$ satisfies the well-known \emph{bounded top-width axiom schema} of   index $2$, cf.\ \cite[p.\ 112]{CZ97}, which is refuted by intuitionistic logic. The problem of determining which superintuitionistic logics are definable by classes of polyhedra is open; e.g.,  what is the logic of the class $\P$ of all closed triangulable manifolds?\qed
\end{rem}

\section*{Acknowledgements}
The first-named author was  partially supported by Shota Rustaveli National Science Foundation grant \#DI-2016-25. The remaining authors were partially supported by the Italian FIRB "Futuro in Ricerca" grant \#RBFR10DGUA.

\providecommand{\bysame}{\leavevmode\hbox to3em{\hrulefill}\thinspace}
\providecommand{\href}[2]{#2}

\end{document}